\documentclass[11pt,reqno]{amsart}

\usepackage{marginnote,amssymb,mathrsfs,latexsym,verbatim,pdfsync,color,mathabx,enumerate}

\newcommand{\bbC}{{\mathbb C}}
\newcommand{\bbD}{{\mathbb D}}
\newcommand{\bbH}{{\mathbb H}} 

\newcommand{\bbR}{{\mathbb R}}

\newcommand{\bbZ}{{\mathbb Z}}

\def\Re{\operatorname{Re}}
\def\sgn{\operatorname{sgn}}
\def\jac{\operatorname{Jac}}

\def\z{\zeta} 

\def\p{\partial}

\def\Hol{\operatorname{Hol}}
\def\Aut{\operatorname{Aut}}

\newtheorem{thm}{Theorem}[section]
\newtheorem{prop}[thm]{Proposition}

\newtheorem{lem}[thm]{Lemma}

\newtheorem{remark}[thm]{Remark}

\newtheorem{theorem}{Theorem}

\begin{document}

\title[Holomorphic function spaces on the Hartogs triangle]{Holomorphic function spaces on the Hartogs triangle} 
\author[A. Monguzzi]{Alessandro Monguzzi}
\address{Dipartimento di Matematica e Applicazioni, Universit\`a degli Studi di
  Milano--Bicocca, Via R. Cozzi 55, 20126 Milano, Italy}
\email{{\tt alessandro.monguzzi@unimib.it}}
\thanks{{\em Math Subject Classification} 31C25; 32A07; 32A25; 32A36; 32A50}
\thanks{The author is partially supported by the 2015 PRIN grant
  \emph{Real and Complex Manifolds: Geometry, Topology and Harmonic analysis}  
  of the Italian Ministry of Education (MIUR) and by the Progetto GNAMPA (INdAM) 2020 “Alla frontiera tra l’analisi complessa in piu’
variabili e l’analisi armonica”.}

\begin{abstract}
The definition of classical holomorphic function spaces such as the Hardy space or the Dirichlet space on the Hartogs triangle is not canonical. In this paper we introduce a natural family of holomorphic function spaces on the Hartogs triangle which includes some weighted Bergman spaces,  a candidate Hardy space and a candidate Dirichlet space. For the weighted Bergman spaces and the Hardy space we study the $L^p$ mapping properties of Bergman and Szeg\H o projection respectively, whereas for the Dirichlet space we prove it is isometric to the Dirichlet space on the bidisc.
\end{abstract}
\maketitle

The Hartogs triangle 
\begin{equation*}
\bbH=\big\{(z_1,z_2)\in\mathbb C^2: |z_1|<|z_2|<1\big\}
\end{equation*}
is a peculiar domain which is a source of many pathological facts in complex analysis (\cite{Shaw}). For instance, it is not hard to prove that $\bbH$ is a domain of holomorphy which does not admit a Stein neighborhood basis, that is, it is not true that $\overline \bbH=\cap_k \Omega_k$, where each $\Omega_k$ is a pseudoconvex domain containing $\bbH$. The domain $\bbH$ is biholomorphically equivalent to the product domain $\mathbb D\times\bbD^*$, where $\bbD$ denotes the unit disc in the complex plane and $\bbD^*$ denotes the punctured disc, by means of the map $\Phi: \mathbb D\times\bbD^*\to\bbH$, $(w_1,w_2)\mapsto (w_1w_2,w_2)$. 

Recently, several mathematicians careful studied the mapping properties of the Bergman projection and related problems both for the Hartogs triangle and for certain its generalizations. In \cite{CZ} the $L^p$ mapping properties for the classical Hartogs triangle are completely characterized; in \cite{HW2, HW} the authors exploit techniques of dyadic harmonic analysis to investigate the mapping properties of the Bergman projection on weighted $L^p$ spaces and they also prove some endpoint estimates; in \cite{EMcN,E-thesis,E-pacific, EMcN17,CEMcN,EdMc-sob} the authors introduce some generalizations of the Hartogs triangle, called \emph{thin} or \emph{fat} Hartogs triangles, and they study several problems such as the $L^p$ and Sobolev mapping properties of the Bergman projection and the zeros of the Bergman kernel; in
\cite{Khanh} Bergman-Toeplitz operators on the fat Hartogs triangle are studied. In \cite{Chen1,Chen2,Chen3,Chen4,Chen5} some other generalizations of the Hartogs triangle and associated weighted Bergman projections are investigated. 

In most of the cited papers the peculiar geometry of the Hartogs triangle, or of its generalizations, forces the Bergman projection to be unbounded on some $L^p$ spaces. However, the exact relationship between the geometry of the domain and the regularity of the Bergman projection is not fully understood in generality. A similar phenomenon, where the geometry of the domain forces an irregular behavior of the Bergman projection, were already observed in the setting of the worm and model worm domains \cite{Bar,P1,P2,Bar2,P3,P4}. Analogous results for the Szeg\H o projection, were discussed and proved in the setting of model worm domains in \cite{M1,M2,M3,M4,M5,LS} and in other setting in  \cite{BB95, LS04, MZ15}.

It would be interesting to study on $\bbH$  and its generalizations other classical function spaces besides the Bergman and weighted Bergman spaces. Unfortunately, because of the peculiar geometry, the definition of classical spaces such the Hardy space or the Dirichlet space is not canonical on $\bbH$.  The main novelty in this paper is the introduction of a family of holomorphic function spaces which includes a candidate Hardy space and a candidate Dirichlet space. Our spaces depend on a real parameter  $\nu\in[-2,+\infty)$; for $\nu>-1$ we deal with some weighted Bergman spaces $A^2_\nu(\bbH)$, the case $\nu=-1$ identifies our candidate Hardy space $H^2(\bbH)$, we have some weighted Dirichlet spaces $\mathcal D_{\nu}$ for $-2<\nu<-1$ and we have our candidate Dirichlet space $\mathcal D(\bbH)$ for $\nu=-2$.

Briefly, for $\nu>-1$ and $1<p<\infty$, the weighted Bergman spaces $A^p_{\nu}(\bbH)$ are defined as
\begin{equation*}
A^p_\nu(\bbH)= L^p_{\nu}(\bbH)\cap \Hol(\bbH)
\end{equation*}
where $L^p_\nu(\bbH)$ denotes the weighted $L^p$ spaces endowed with the norm
\begin{equation}\label{Lp-norm}
\|f\|^p_{L^p_\nu}:=C_\nu\int_{\bbH}|f(z_1,z_2)|^p|z_2|^{\nu}(1-|z_1/z_2|^2)^\nu(1-|z_2|^2)^{\nu}\, dz.
\end{equation}
Here $dz$ is the Lebesgue measure in $\bbC^2$ and $C_\nu$ is a normalization constant. We refer the reader to Section \ref{weighted-bergman} for more comments and motivations about the definition of the spaces $A^2_{\nu}(\bbH)$. The weighted Bergman projection $P_\nu$ is the orthogonal projection from $L^2_\nu(\bbH)$ onto $A^2_\nu(\bbH)$. Let $\lfloor x\rfloor$ be the floor function. We prove the following result.

\begin{theorem}\label{main-1}
 Let $\nu>-1$ and let $P_\nu$ be the weighted Bergman projection densely defined on  $L^2_\nu(\bbH)\cap L^p_\nu(\bbH)$ for $p\in (1,+\infty)$. Then, we have the following:
 \begin{enumerate}[(i)]
 \item if $\nu>0$ and $\nu\neq 2n,n\in\mathbb N$, the weighted Bergman projection $P_\nu$ extends to a bounded operator $P_\nu: L^p_\nu(\bbH)\to L^p_\nu(\bbH)$ if and only if $p\in \big(2-\frac{\nu-2\lfloor\nu/2\rfloor}{2+\nu-\lfloor\nu/2\rfloor}, 2+\frac{\nu-2\lfloor\nu/2\rfloor}{2+\lfloor\nu/2\rfloor}\big)$;
 \medskip
 
 \item if $\nu=2n,n \in\mathbb N_0$, the weighted Bergman projection $P_\nu$ extends to a bounded operator $P_\nu: L^p_\nu(\bbH)\to L^p_\nu(\bbH)$ if and only if  $p\in \Big(2-\frac{2}{3+n}, 2+\frac 2{1+n}\Big)$;
 
 \medskip
 
 \item if $-1<\nu<0$,  the weighted Bergman projection $P_\nu$ extends to a bounded operator $P_\nu: L^p_\nu(\bbH)\to L^p_\nu(\bbH)$ if and only if $p\in \Big(2-\frac{2+\nu}{3+\nu}, 4+\nu\Big)$.
\end{enumerate}
\end{theorem}

\medskip

A comparison between Theorem \ref{main-1} and the results in \cite{Chen2} is needed. In \cite[Theorem 1]{Chen2} the $L^p$ mapping properties of the weighted Bergman projection $\widetilde P_\nu$ associated to the weighted Bergman spaces defined with respect to the measure $|z_2|^\nu\, dz,  \nu\in\mathbb R$, are completely characterized. If we restrict to the case $\nu>-1$, the projections $\widetilde P_\nu$ and $P_\nu$ share the same $L^p$ mapping properties. The extra factor $(1-|z_1/z_2|^2)^2(1-|z_2|^2)^\nu$ in the measure we use to define the spaces $A^2_\nu(\bbH)$ does not influence the mapping properties of the weighted Bergman projection. However, the weighted Bergman spaces studied in \cite{Chen2} are meaningful for any $\nu\in\mathbb R$, whereas our weighted Bergman spaces turn out to be trivial for $\nu\leq-1$ (Remark \ref{remark}). Hence, we introduce the weighted Dirichlet spaces.

\medskip

We refer the reader to Section \ref{hardy-space} for a precise definition of the Hardy space $H^2(\bbH)$ and for some motivation on why it corresponds to the case $\nu=-1$. Here we only say that $H^2(\bbH)$ can be described as the space of functions of the form
\begin{equation*}
f(z_1,z_2)=\sum_{j=0}^{+\infty}\sum_{k=-j-1}^{+\infty}a_{jk}z_1^jz_2^k
\end{equation*}
endowed with the norm
\begin{equation*}
\|f\|^2_{H^2}:=\sum_{j=0}^{+\infty}\sum_{k=-j-1}^{+\infty}|a_{jk}|^2.
\end{equation*}

The space $H^2(\bbH)$ can be identified with a closed subspace $H^2(d_b(\bbH))$ of $L^2(d_b(\bbH))$, the space of square-integrable functions on the distinguished boundary $d_b(\bbH)$ of the Hartogs triangle $\bbH$. The boundary $d_b(\bbH)$ coincides with $\p\bbD\times\p \bbD$, where $\bbD$ is the unit disc and $\p\bbD$ its topological boundary (see, for instance, \cite[Theorem 1.9]{Chen-PAMQ}). The Szeg\H o projection associated to $H^2(\bbH)$ is the Hilbert space projection operator $S:L^2(d_b(\bbH))\to H^2(d_b(\bbH))$. We obtain the following result.

\begin{theorem}\label{main-2}
 The Szeg\H o projection $S$ densely defined on $L^2(d_b(\bbH))\cap L^p(d_b(\bbH))$ extends to a bounded operator $S:L^p(d_b(\bbH))\to L^p(d_b(\bbH))$ for any $p\in(1,+\infty)$.
\end{theorem}
In light of Theorem \ref{main-1} this result about the Szeg\H o projection is surprising and unexpected. As we will see in detail in Section \ref{hardy-space}, the space $H^2(\bbH)$ is, in a suitable sense, the limit space of the weighted Bergman spaces $A^2_\nu(\bbH)$ for $\nu\to-1$. The space $H^2(\bbH)$ turns out to be modeled only on the distinguished boundary $d_b(\bbH)$ of $\bbH$ and not on the whole topological boundary. Therefore, we loose sight of the origin $(0,0)$, which is the most pathological point of the topological boundary of $\bbH$. 

Finally, the Dirichlet space $\mathcal D(\bbH)$ can be described in short as the space of functions of the form 
\begin{equation*}
f(z_1,z_2)=\sum_{j=0}^{+\infty}\sum_{k=-j}^{+\infty}a_{jk}z_1^jz_2^k
\end{equation*}
endowed with the norm
\begin{equation*}
\|f\|^2_{\mathcal D}:=\sum_{j=0}^{+\infty}\sum_{k=-j}^{+\infty}(j+1)(j+k+1)|a_{jk}|^2.
\end{equation*}
In Section \ref{dirichlet-space} we motivate the definition of this space and we see why it corresponds to the value $\nu=-2$ in our family of spaces. The following theorem provide a motivation of why the space $\mathcal D$ is a candidate Dirichlet space on $\bbH$.
\begin{theorem}\label{main-3}
 There exists a surjective isometry from $\mathcal D(\bbH)$ onto $\mathcal D(\bbD\times\bbD)$, the Dirichlet space on the bidisc.
\end{theorem}

All the function spaces we consider are reproducing kernel Hilbert spaces with kernel $K_{\nu}$ which satisfies the estimate
\begin{equation}\label{kernel-estimate}
|K_{\nu}((z_1,z_2),(w_1,w_2))|\leq c_\nu |z_2\overline w_2|^{-1-\lceil\frac\nu2\rceil}|1-(z_1\overline w_1)/(z_2\overline w_2)|^{-(\nu+2)}|1-z_2\overline w_2|^{-(\nu+2)}
\end{equation}
for any $\nu>-2$ and $((z_1,z_2),(w_1,w_2))\in\bbH\times\bbH$ and for some positive constant $c_\nu$. Here $\lceil\cdot\rceil$ denotes the ceiling function. Because of this estimate we say that our family of spaces can be considered an analogue of a classical family of function spaces on the unit disc (or  the unit ball). If we consider the case of the unit disc, a well-known family of function spaces is the family of reproducing kernel Hilbert spaces associated to the family of kernels
\begin{equation}\label{kernel-disc}
k_\nu(z,w)=\frac{1}{(1-z\overline w)^{2+\nu}},
\end{equation}
where $\nu>-2$. This family of spaces is known  in the literature with several names, such as Besov-Sobolev spaces or Dirichlet type spaces. We point out that for $\nu>-1$, the kernels $k_\nu$'s identify the classical weighted Bergman spaces on the unit disc, the case $\nu=-1$ is the case of the Hardy space, the case $-2<\nu<-1$ identifies the weighted Dirichlet spaces, whereas the case $\nu\to-2$ identifies the Dirichlet space and a logarithmic singularity arises in \eqref{kernel-disc}. For more about these spaces we refer the reader to \cite{Wu,Zhu, ARSB}.

\medskip

The paper is organized as follows. In Section \ref{weighted-bergman} we introduce the weighted Bergman spaces $A^2_\nu(\bbH)$ and we prove Theorem \ref{main-1}. In Section \ref{hardy-space} we deal with the space $H^2(\bbH)$ and we prove Theorem \ref{main-2}. In Section \ref{dirichlet-space} we introduce the spaces $\mathcal D_\nu(\bbH)$ and $\mathcal D(\bbH)$ and we prove Theorem \ref{main-3}. We conclude with some final remarks in Section \ref{final-remarks}.

\medskip

Given two positive quantities $X,Y$ we write $X\approx Y$ meaning that there exist two positive constants $c,C$
 such that $cX\leq Y\leq CX$, whereas we write $X\lesssim Y$ if there exists a positive constant $C$ such that $X\leq CY$.

\section{Weighted Bergman spaces
}\label{weighted-bergman}

Let $\mathbb D$ denote the unit disc in the complex plane and let $H^2(\bbD)$ be the classical Hardy space, that is, the space of holomorphic functions on $\bbD$ such that 
\begin{equation*}
\|f\|^2_{H^2(\bbD)}:=\sup_{0<r<1}\frac{1}{2\pi}\int_0^{2\pi}|f(re^{i\theta})|^2\, d\theta<+\infty.
\end{equation*}
Given a function $f\in H^2(\bbD)$ it is well--known  (see, for instance, \cite{zhu-t}) that 
\begin{equation}\label{hardy-limit}
\|f\|^2_{H^2(\bbD)}=\lim_{\nu\to-1^+}\|f\|^2_{A^2_{\nu}(\bbD)},
\end{equation}
where $A^2_\nu(\bbD)$ denotes the weighted Bergman space on $\bbD$ with norm
\begin{equation}\label{norm-disc}
\|f\|^2_{A^2_\nu(\bbD)}=(\nu+1)\int_{\bbD}|f(z)|^2(1-|z|^2)^{\nu}\, dz.
\end{equation}

Given $K_{\bbD}(z,w)=(1-\overline zw)^{-2}$, the reproducing kernel of the unweighted Bergman space $A^2(\bbH)$, notice that the weight in \eqref{norm-disc} satisfies
\begin{equation*}
(1-|z|^2)^{\nu}= K^{-\frac\nu2}_{\bbD}(z,z)\approx \delta^{\nu}(z),
\end{equation*}
where $\delta(z)$ is the distance of $z\in\bbD$ from the topological boundary $\partial\bbD$. 

In light of the above, for $\nu>-1$ we define $A^2_\nu(\bbH)$ as the space of holomorphic functions on $\bbH$ such that 
\begin{equation}\label{weighted-norm}
 \|f\|^2_{A^2_\nu(\bbH)}:=C_\nu\int_{\bbH}|f(z)|^2 K^{-\frac\nu2}(z,z)\, dz<+\infty,
\end{equation}  
where $z=(z_1,z_2)$, $K(z,w)$ is the reproducing kernel of the unweighted Bergman space $A^2(\bbH)$, $dz$ denotes the Lebesgue measure of $\bbC^2$ and $C_\nu$ is a normalization constant to be determined.  In Section \ref{hardy-space} we will define the Hardy space $H^2(\bbH)$ as the limit of the $A^2_\nu(\bbH)$ spaces as $\nu\to-1$ in the sense of \eqref{hardy-limit}.

The unweighted Bergman space $A^2(\bbH)$ corresponds to the case $\nu=0$ but we drop the subscript in the notation $A^2_0(\bbH)$. In this case we set 
\begin{equation*}
\|f\|^2_{A^2(\bbH)}=\frac{2}{\pi^2}\int_{\bbH}|f(z_1,z_2)|^2\, dz,
\end{equation*}
that is, we fix the constant $C_\nu=C_0$ in \eqref{weighted-norm} to be $2/\pi^2$. With such a normalization it is easily verified that $\{z_1^jz_2^k: j\geq0, j+k+1\geq0\}$ is an orthogonal basis for $A^2(\mathbb H)$ and that 
\begin{equation*}
\|z_1^jz_2^k\|^2_{A^2}=\frac{2}{(j+1)(k+j+2)}.
\end{equation*}

We recall that if $\mathcal H$ denotes a reproducing kernel Hilbert space and $\{\varphi_n\}_{n}$ is an orthonormal basis for $\mathcal H$, then the reproducing kernel $K_{\mathcal H}$ is given by the sum $K_{\mathcal H}=\sum_{n}\varphi_n\overline\varphi_{n}$. Therefore, the Bergman kernel of $\bbH$ is given by
\begin{align*}
 \begin{split}
  K(z,w)&=\frac{1}{2}\sum_{j=0}^{+\infty}\sum_{k=-j-1}^{+\infty}(j+1)(k+j+2)(z_1\overline w_1)^j(z_2\overline w_2)^k=\frac{1}{2 z_2\overline w_2 (1-\frac{z_1\overline w_1}{z_2\overline w_2})^2(1-z_2\overline w_2)^2}.
   \end{split}
\end{align*}
In the general case $\nu>-1$ we set
\begin{equation}\label{mu}
d\mu_\nu(z):= C_\nu K^{-\frac\nu 2}(z,z)\, dz= C_\nu 2^\frac\nu2 |z_2|^\nu(1-|z_1/z_2|^2)^\nu(1-|z_2|^2)^{\nu}\, dz
\end{equation}
where 
\begin{equation}\label{Cnu}
C_\nu=\frac{(\nu+1)\Gamma(\frac32\nu+3)}{2^{\frac\nu2}\pi^{2}\Gamma(\nu+1)\Gamma(\frac\nu2+2)}.
\end{equation}
Notice that,
\begin{align*}
 \int_\bbH K^{-\frac\nu 2}(z,z)\, dz&= 2^\frac\nu2\int_{\bbH}|z_2|^{\nu}(1-|z_1/z_2|^2)^{\nu}(1-|z_2|^2)^{\nu}\, dz\\
 &=2^{\frac\nu2}4\pi^2\int_0^1\int_0^\rho r\rho^{\nu+1}(1-(r/\rho)^2)^\nu(1-\rho^2)^\nu\, drd\rho\\
 &=2^{\frac\nu2}4\pi^2\int_0^1\int_0^1 r\rho^{\nu+3}(1-(r/\rho)^2)^\nu(1-\rho^2)^\nu\, drd\rho\\
 &=2^\frac\nu2\pi^2\frac{\Gamma(\nu+1)\Gamma(\frac\nu2+2)}{(\nu+1)\Gamma(\frac32\nu+3)},
\end{align*}
where the last equality follows from from the identity
\begin{equation*}
\int_0^1 (1-x)^a x^b\,dx=\frac{\Gamma(a+1)\Gamma(b+1)}{\Gamma(a+b+2)}
\end{equation*}
whenever $\Re a,\Re b>-1$. Thus, we deduce from \eqref{mu} and \eqref{Cnu} that
\begin{equation*}
\int_{\mathbb H}\,d\mu_{\nu}(z)=1
\end{equation*}
for all $\nu>-1$.
In conclusion, the spaces $A^2_\nu(\bbH)$ are defined as
\begin{equation*}
A^2_\nu(\bbH)=\Big\{f\in \Hol(\bbH): \|f\|^2_{A^2_\nu}:=\int_{\bbH}|f(z_1,z_2)|^2\, d\mu_\nu(z)<+\infty\Big\}.
 \end{equation*}

\begin{remark}\label{remark}\emph{
If $\nu\leq-1$ we have $\int_\bbH d\mu_\nu(z)=+\infty$ and the space $A^2_\nu(\bbH)$ is empty. In fact, $\|z_1^jz_2^k\|_{A^2_\nu}=+\infty$ for any non-negative integer $j$ and any integer $k$.  Hence, exploiting the uniform convergence of the Laurent expansion $f=(z_1,z_2)=\sum_{j=0}^{+\infty}\sum_{k=-\infty}^{+\infty}a_{jk}z_1^jz_2^k
$
of any function $f\in\Hol(\bbH)$, we conclude that $f$ does not belong to any $A^2_\nu(\bbH)$. Hence, $A^2_\nu(\bbH)=\emptyset$.}
\end{remark}

\medskip

We now recall the following proposition, proved in \cite{E-thesis}, about the diagonal behavior of the kernel $K$.
\begin{prop}[\cite{E-thesis}] The following facts hold true.
\begin{enumerate}
             \item[$(i)$] Let $\delta(z)$ be the distance  of $z$ to $\partial\bbH$, the topological boundary of $\bbH$. Then, 
             \begin{equation*}
             K(z,z)\approx \delta(z)^{-2}
             \end{equation*}
             as $z$ tends to the origin.
\medskip
\item[$(ii)$] Let $p$ be any point in the distinguished boundary $d_b(\bbH)$.   For any number $\beta\in (2,4]$ there exists a path $\gamma: [1/2,1]\to\overline\bbH$ such that $\gamma(1)=p$ and for all $u\in[1/2,1)$,
\begin{equation*}
K(\gamma(u),\gamma(u))\approx \delta(\gamma(u))^{-\beta}.
\end{equation*}
\end{enumerate}

\end{prop}
\begin{proof}
We refer the reader to Theorems 3.1.4 and 3.1.5 in \cite{E-thesis}.
\end{proof}

We want to prove that the spaces $A^2_\nu(\bbH)$ are reproducing kernel Hilbert spaces and study the regularity of the associated weighted Bergman projection. We need the following elementary lemma. 

\begin{lem} \label{monomial-norm}
Let $\nu>-1$. Then, the monomial $z_1^jz_2^k$ belongs to $A^2_\nu$ if and only if
\begin{equation*}
j\geq0 \quad\text{and}\quad j+k+\frac\nu2+2>0.
\end{equation*}
\end{lem}
\begin{proof}
From \eqref{mu}, integrating in polar coordinates, we obtain
\begin{align*}
 \|z_1^jz_2^k\|^2_{A^2_{\nu}}&=2^\frac\nu2 C_\nu\int_{\bbH}|z_1|^{2j}|z_2|^{2k} |z_2|^{\nu}(1-|z_1/z_2|^2)^\nu(1-|z_2|^2)^\nu\,dz\\
 &=2^{\frac\nu2+2}\pi^{2}C_\nu \int_0^1\int_0^\rho r^{2j+1}\rho^{2k+\nu+1}(1-(r/\rho)^2)^\nu(1-\rho^2)^\nu\, drd\rho\\
 &=2^{\frac\nu2+2}\pi^2 C_\nu\int_0^1\int_0^1 r^{2j+1}\rho^{2(j+k)+\nu+3}(1-r^2)^{\nu}(1-\rho^2)^{\nu}\, drd\rho\\
 &=\frac{(\nu+1)\Gamma(\nu+1)\Gamma(\frac32\nu+3)}{\Gamma(\frac\nu2+2)}\frac{\Gamma(j+1)\Gamma(j+k+\frac\nu2+2)}{\Gamma(j+\nu+2)\Gamma(j+k+\frac32\nu+3)},
 \end{align*}
 where the last equality holds true if and only if
 \begin{equation*}
 j\geq 0 \quad\text{and}\quad j+k+\frac\nu2+2>0
 \end{equation*}
 as we wished to show.
 \end{proof}
From the previous lemma we conclude that any function $f\in A^2_\nu(\bbH)$ is of the form
 \begin{equation}\label{expansion-f}
  f(z_1,z_2)=\sum_{j=0}^{+\infty}\sum_{k>-j-\frac\nu2-2}a_{jk}z_1^jz_2^k
 \end{equation}
and
\begin{align}\label{expansion-norm}
\begin{split} 
 \|f\|^2_{A^2_\nu}=\frac{(\nu+1)\Gamma(\nu+1)\Gamma(\frac32\nu+3)}{\Gamma(\frac\nu2+2)}\sum_{j=0}^{+\infty}\sum_{k>-j-\frac\nu2-2}\frac{\Gamma(j+1)\Gamma(j+k+\frac\nu2+2)}{\Gamma(j+\nu+2)\Gamma(j+k+\frac32\nu+3)}|a_{jk}|^2
 \end{split}
\end{align}

Moreover, we have the following proposition.
\begin{prop}\label{Bergman-kernel} Let $\nu>-1$. The following properties hold.
\begin{enumerate}
\item[$(i)$] The space $A^2_\nu(\bbH)$ is a reproducing kernel Hilbert space.
\medskip

\item[$(ii)$] An orthonormal basis for $A^2_\nu(\bbH)$ is given by $
 \{z_1^jz_2^k/\|z_1^jz_2^k\|_{A^2_\nu}\}_{(j,k)\in I_{\nu}},
 $
 where $  I_{\nu}=\big\{(j,k): j\geq0, j+k+\frac\nu2+2>0\big\}$.
\medskip

\item[$(iii)$] The reproducing kernel $K_\nu$ of $A^2_{\nu}(\bbH)$ is given by
\begin{equation*}
K_{\nu}(z,w)=\frac{\Gamma(\frac\nu2+2)}{\Gamma(\frac32\nu+3)}\frac{(z_2\overline w_2)^{-2}}{(1-\frac{z_1\overline w_1}{z_2\overline w_2})^{\nu+2}}\sum_{k>-\frac\nu2}^\infty\frac{\Gamma(k+\frac32\nu+1)}{\Gamma(k+\frac\nu2)}(z_2\overline w_2)^k.
\end{equation*}
\end{enumerate}
\end{prop}
\begin{proof}
Let $I_{\nu}=\big\{(j,k): j>0, j+k+\frac\nu2+2>0\big\}$ and set
\begin{equation*}
\nu_{jk}:=\frac{\Gamma(j+1)\Gamma(j+k+\frac\nu2+2)}{\Gamma(j+\nu+2)\Gamma(j+k+\frac32\nu+3)},
\end{equation*}
so that, from \eqref{expansion-norm}, we have
\begin{equation*}
\|f\|^2_{A^2_\nu}\approx\sum_{(j,k)\in I_\nu} \nu_{jk}|a_{jk}|^2.
\end{equation*}
Then, for any $z=(z_1,z_2)\in\bbH$,
\begin{align*}
 |f(z_1,z_2)|&\leq \sum_{(j,k)\in I_\nu} |a_{jk}||z_1|^j|z_2|^k\leq \Big(\sum_{(j,k)\in I_\nu} |a_{jk}|^2 \nu_{jk}\Big)^{\frac12}\Big(\sum_{(j,k)\in I_\nu} |z_1|^{2j}|z_2|^{2k} \nu_{jk}^{-1}\Big)^{\frac12}.
\end{align*}
Since $(z_1,z_2)$ is a fixed point in $\bbH$ the latter factor in the above estimate is finite, whereas the first factor is comparable to the $A^2_\nu(\bbH)$ norm of the function $f$. Therefore, if $c_{\nu,z}$ denotes a positive constant depending on $\nu$ and $z=(z_1,z_2)\in\bbH$, we obtain
\begin{equation*}
|f(z_1,z_2)|\leq c_{\nu,z}\|f\|_{A^2_\nu},
\end{equation*}
that is, the point-evaluation functionals are bounded on $A^2_\nu(\bbH)$. Hence, the space $A^2_{\nu}(\bbH)$ is a reproducing kernel Hilbert space. This concludes the proof of $(i)$. The proof of $(ii)$ follows at once from Lemma \ref{monomial-norm} and \eqref{expansion-f}. Consequently, we obtain that the reproducing kernel $K_\nu$ of $A^2_\nu(\bbH)$ is given by
\begin{align}\label{kernel}
\begin{split}
 K_\nu(z,w)&=K_\nu((z_1,z_2),(w_1,w_2))\\
 &=\frac{\Gamma(\frac\nu2+2)}{(\nu+1)\Gamma(\nu+1)\Gamma(\frac32+3)}\sum_{(j,k)\in I_\nu}\frac{\Gamma(j+\nu+2)\Gamma(j+k+\frac32\nu+3)}{\Gamma(j+1)\Gamma(j+k+\frac\nu2+2)}(z_1\overline w_1)^j(z_2\overline w_2)^k\\
&=\frac{\Gamma(\frac\nu2+2)}{\Gamma(\frac32\nu+3)}\frac{(z_2\overline w_2)^{-2}}{(1-\frac{z_1\overline w_1}{z_2\overline w_2})^{\nu+2}}\sum_{k>-\frac\nu2}^\infty\frac{\Gamma(k+\frac32\nu+1)}{\Gamma(k+\frac\nu2)}(z_2\overline w_2)^k
\end{split}
\end{align}
as we wished to show. The last equality follows from the identity
\begin{equation*}
\frac{1}{(1-z)^\lambda}=\frac{1}{\Gamma(\lambda)}\sum_{n=0}^\infty\frac{\Gamma(n+\lambda)}{n!}z^n,
\end{equation*}
which holds for $\Re \lambda>0$ and $|z|<1$. This concludes the proof of $(iii)$ and of the proposition.  
\end{proof}
\begin{remark} \emph{
A closed formula for the kernel $K_\nu$ is immediately deduced. Given $\alpha,\beta$ and $\gamma$, real or complex parameters where $\gamma$ is not a non-positive integer, the Hypergeometric Function is defined as 
\begin{equation*}
F(\alpha,\beta;\gamma;z)= \frac{\Gamma(\gamma)}{\Gamma(\alpha)\Gamma(\beta)}\sum_{k=0}^{+\infty}\frac{\Gamma(k+\alpha)\Gamma(k+\beta)}{\Gamma(k+\gamma)k!}z^k.
\end{equation*}
It is easily seen that the radius of convergence of this series is $1$. For more properties of the Hypergeometric Function we refer the reader, for instance, to \cite{Lebedev}.
From \eqref{kernel} we get that
\begin{equation}\label{kernel-general}
K_{\nu}(z,w)=a_\nu \frac{(z_2\overline w_2)^{-1-\lceil\frac\nu2\rceil}}{(1-\frac{z_1\overline w_1}{z_2\overline w_2})^{\nu+2}}F\Big(\frac32\nu-\Big\lceil\frac\nu2\Big\rceil+2,1;\frac\nu2-\Big\lceil\frac\nu2\Big\rceil+1;z_2\overline w_2\Big)
\end{equation}
where $a_\nu=\frac{\Gamma(\frac\nu2+2)\Gamma(\frac32\nu-\lceil\frac\nu2\rceil+2)}{\Gamma(\frac32\nu+3)\Gamma(\frac\nu2-\lceil\frac\nu2\rceil+1)}$. Notice that when $\nu=2n,n\in\mathbb N_0$ it holds
\begin{equation*}
F\Big(\frac32\nu-\lceil\frac\nu2\rceil+2,1;\frac\nu2-\lceil\frac\nu2\rceil+1;z_2\overline w_2\Big)=(1-z_2\overline w_2)^{-2n-2}.
\end{equation*}
}
\end{remark}

\medskip

Given the weighted Bergman spaces $A^2_\nu(\bbH)$ it is a natural question to investigate the $L^p$ mapping properties of the associated weighted Bergman projection, that is, the operator
\begin{equation*}
P_\nu f(z_1,z_2):=\int_{\bbH} f(w_1,w_2) K_{\nu}((w_1,w_2),(z_1,z_2))\, d\mu_\nu(w),
\end{equation*}
where $d\mu_{\nu}$ is the measure defined in \eqref{mu}. We recall that $P_\nu$ is the orthogonal projection of $L^2_\nu(\bbH)$ onto $A^2_\nu(\bbH)$, where $L^2_\nu(\bbH)$ is the space of square-integrable functions with respect to the measure $d\mu_\nu$.

We now prove the necessary part of Theorem \ref{main-1}. 

\proof[Necessary condition of Theorem \ref{main-1}]
Let us consider the function $f(z_1,z_2)=\overline z_2^{1+\lceil\frac\nu2\rceil}$ which belongs to $L^p_\nu$ for any $p\in(1,\infty)$. From the Laurent series of the kernel \eqref{kernel} we deduce
\begin{align*}
 P_\nu f(z_1,z_2)= d_\nu z_2^{-1-\lceil\nu/2\rceil}
\end{align*}
for some positive constant $d_\nu$. Now,
 \begin{align*}
  \|P_\nu f\|^p_{L^p_\nu}&\approx \int_{\bbH}|z_2|^{\nu-(1+\lceil\nu/2\rceil)p}(1-|z_1/z_2|^{2})^{\nu}(1-|z_2|^2)^{\nu}\, dz\\
  &\approx\int_0^1\int_0^1 r\rho^{\nu-(1+\lceil\nu/2\rceil)p+3}(1-r^2)^\nu(1-\rho^2)^\nu\, drd\rho,
 \end{align*}
and this last integral diverges if $p\geq\frac{4+\nu}{1+\lceil\nu/2\rceil}=2+\frac{\nu-2\lceil\nu/2\rceil+2}{1+\lceil\nu/2\rceil}$. Therefore, $P_\nu$ cannot extend to a bounded operator if $p\geq 2+\frac{\nu-2\lceil\nu/2\rceil+2}{1+\lceil\nu/2\rceil}$. By a standard duality argument we also obtain that $P_\nu$ cannot be bounded if $1<p\leq2-\frac{\nu-2\lceil\nu/2\rceil+2}{\nu-\lceil\nu/2\rceil+3}$. Indeed, if $1/p+1/p'=1$, since $P_\nu$ is self-adjoint, we get
\begin{align*}
 \|P_\nu f\|_{L^{p'}_\nu}&=\sup_{\|g\|_{L^p_\nu}=1}|\left< P_\nu f,g\right>|=\sup_{\|g\|_{L^p\nu}=1}|\left<f,P_\nu g\right>|\leq \sup_{\|g\|_{L^p_\nu}=1}\|P_\nu g\|_{L^p_\nu}\|f\|_{L^{p'}_\nu}.
\end{align*}
Therefore, the $L^p_\nu$-boundedness of $P_\nu$ would imply the $L^{p'}_{\nu}$-boundedness. Hence, $P_\nu$ cannot be bounded if $p\notin\Big(2-\frac{\nu-2\lceil\frac\nu2\rceil+2}{\nu-\lceil\frac\nu2\rceil+3}, 2+\frac{\nu-2\lceil\frac\nu2\rceil+2}{1+\lceil\frac\nu2\rceil}\Big)$. Using the identity $\lfloor\frac\nu2\rfloor=\lceil\frac\nu2\rceil-1$ it is easy to see that this condition coincides with the conditions \emph{(i),(ii)} and \emph{(iii)} as $\nu$ varies. The proof is concluded.
\qed 

\medskip

The sufficient condition in Theorem \ref{main-1} will be proved by means of the classical Schur's lemma, which we now recall. For a proof of this result we refer the reader, for instance, to \cite[Appendix A]{Grafakos}.

\begin{lem}[Schur's lemma]\label{schur}
Let $(\mathcal X, d\mu_\mathcal X)$,$ (\mathcal Y, d\mu_\mathcal Y )$ be two $\sigma$-finite measure spaces. Let $T$ the integral operator given by
\begin{equation*}
Tf(x)=\int_{\mathcal Y}K(x,y)f(y)\, d\mu_\mathcal Y(y),
\end{equation*}
where $K$ is a measurable positive kernel on $\mathcal X\times\mathcal Y$. Let $1<p,p'<+\infty$ be such that $1/p+1/p'=1$. Suppose there exist positive functions $\psi:\mathcal Y\to (0,+\infty)$, $\varphi: \mathcal X\to (0,+\infty)$ such that
\begin{enumerate}[(i)]
 \item$ \int_\mathcal Y K(x,y)\psi^{p'}\, d\mu_\mathcal Y(y)\leq C\varphi(x)^{p'}$;
 \item $\int_{\mathcal X} K(x,y)\varphi^p(x)\, d\mu_{\mathcal X}(x)\leq C\psi(y)^p$.
\end{enumerate}
Then, $T: L^p(\mathcal Y)\to L^p(\mathcal X)$ is bounded.
\end{lem}

We also need the following lemmas. The first two lemmas are elementary estimates of which we omit the proofs.

\begin{lem}\label{cl-estimate}
 Let $\tau\in\ (0,+\infty)$ and let $\rho\in (0,1)$. Then,
 \begin{equation*}
 \int_0^{2\pi}\frac{1}{|1-\rho e^{i\theta}|^{1+\tau}}\, d\theta\approx (1-\rho)^{-\tau}.
 \end{equation*}
\end{lem}
\begin{lem}\label{cl-estimate2}
 Let $\gamma>-1$, $\delta\in (0,+\infty)$ and $z\in \mathbb D$. Then,
\begin{equation*}
 \int_{\mathbb D}\frac{(1-|w|^2)^\gamma}{|1-z\overline w|^{2+\gamma+\delta}}\,dw\lesssim (1-|z|^2)^{-\delta}. 
 \end{equation*}
\end{lem}
\begin{lem}\label{kernel-estimate2}
Let $\nu>-1$ and let $z\in\mathbb D$. Then,
\begin{equation*}
\Big|F\Big(\frac32\nu-\Big\lceil\frac\nu2\Big\rceil+2, 1;\frac\nu2-\Big\lceil\frac\nu2\Big\rceil+1;z\Big)\Big|\lesssim |1-z|^{-(\nu+2)}. 
\end{equation*}
\end{lem}
\begin{proof}The proof follows from elementary properties of the Hypergeometric function. From \cite[Chapter 9.5]{Lebedev} we get
\begin{align*}
\Big|F\Big(\frac32\nu-\Big\lceil\frac\nu2\Big\rceil+2, 1;\frac\nu2-\Big\lceil\frac\nu2\Big\rceil+1;z\Big)\Big|&=|1-z|^{-(\nu+2)}\Big|F\Big(-\nu-1, \frac\nu2-\Big\lceil\frac\nu2\Big\rceil;\frac\nu2-\Big\lceil\frac\nu2\Big\rceil+1;z\Big)\Big|\\
&\lesssim |1-z|^{-(\nu+2)}
\end{align*}
where the last bound follows from estimates of $|F(\alpha, \beta; \gamma; z)|$ when $\Re(\gamma-\alpha-\beta)>0$ (\cite[Chapter 9.3]{Lebedev}). 
\end{proof}

\begin{remark}\emph{From the previous lemma and equation \eqref{kernel-general} we deduce that the kernel $K_\nu$ satisfies the main estimate \eqref{kernel-estimate} for any $\nu>-1$.}
 \end{remark}

We are now ready to prove the sufficient condition in Theorem \ref{main-1}.
\proof[Sufficient condition of Theorem \ref{main-1}]
We want to apply Schur's lemma to the positive kernel $|K_{\nu}|$ since the boundedness of the operator with such kernel would imply the boundedness of $P_\nu$. We choose
\begin{equation*}
 f(z_1,z_2)= (1-|z_1/z_2|^2)^{-\alpha}(1-|z_2|^2)^{-\beta}|z_2|^{-\gamma}
 \end{equation*}
 as test function, where $\alpha,\beta,\gamma$ are positive real parameters to be chosen later. Let $(z,w)=((z_1,z_2),(w_1,w_2))$. Then, from Lemma \ref{kernel-estimate2}  and the change of variables $(w_1/w_2,w_2)\mapsto (z_1,z_2)$ we get
\begin{align}\label{estimate-schur}
\begin{split}
\!\!\int_{\bbH}&  |K_\nu(z,w)|f^p(w_1,w_2)\, d\mu_\nu(w)\!\approx\!\int_{\bbH}|K_\nu(z,w)||w_2|^{\nu-\gamma p}(1-|w_1/w_2|^2)^{\nu-\alpha p}(1-|w_2|^2)^{\nu-\beta p} dw\\
&\lesssim |z_2|^{-1-\lceil\frac\nu2\rceil}\int_{\bbH} |w_2|^{\nu-\lceil\frac\nu2\rceil-\gamma p-1}\frac{(1-|w_1/w_2|^2)^{\nu-\alpha p}(1-|w_2|^2)^{\nu-\beta p}}{|1-\frac{z_1\overline w_1}{z_2\overline w_2}|^{2+\nu}|1-z_2\overline w_2|^{2+\nu}}\, dw\\
&\lesssim |z_2|^{-1-\lceil\frac\nu2\rceil}\bigg(\int_{\mathbb D}\frac{(1-|w_1|^2)^{\nu-\alpha p}}{|1- \overline w_1\frac{z_1}{z_2}|^{2+\nu}}\, dw_1\bigg)\bigg(\int_{\bbD^*}\frac{|w_2|^{\nu-\lceil\frac\nu2\rceil-\gamma p+1}(1-|w_2|^2)^{\nu-\beta p}}{|1-z_2\overline w_2|^{2+\nu}}\, dw_2\bigg).
\end{split}
\end{align}
Now, we require $\nu-\alpha p>-1$ and we apply Lemma \ref{cl-estimate2} with $\gamma=\nu-\alpha p$ and $\delta=\alpha p$ to the first integral obtaining
\begin{align*}
\bigg(\int_{\mathbb D}\frac{(1-|w_1|^2)^{\nu-\alpha p}}{|1- \overline w_1\frac{z_1}{z_2}|^{2+\nu}}\, dw_1\bigg)\lesssim (1-|z_1/z_2|)^{-\alpha p}. 
\end{align*}
For the second integral, applying Lemma \ref{cl-estimate} with $\tau=1+\nu$ and requiring the conditions $\nu-\beta p>-1$ and $\nu-\lceil\nu/2\rceil-\gamma p+2>-1$, we have
\begin{align*}
\int_{\mathbb D^*}|w_2|^{\nu-\lceil\frac\nu2\rceil-\gamma p+1}&\frac{(1-|w_2|^2)^{\nu-\beta p}}{|1-z_2\overline w_2|^{2+\nu}}\, dw_2\lesssim\int_0^1 \rho^{\nu-\lceil\frac\nu2\rceil-\gamma p+2}\frac{(1-\rho)^{\nu-\beta p}}{(1-|z_2|\rho)^{1+\nu}}\, d\rho\\
&=\sum_{k=0}^{+\infty}\frac{\Gamma(k+1+\nu)}{\Gamma(\nu+1)\Gamma(k+1)}|z_2|^{k}\int_0^1\rho^{k+\nu-\lceil\frac\nu2\rceil-\gamma p+2}(1-\rho)^{\nu-\beta p}\, d\rho\\
&=\sum_{k=0}^{+\infty}\frac{\Gamma(k+1+\nu)\Gamma(k+\nu-\lceil\frac\nu2\rceil-\gamma p+3)\Gamma(\nu-\beta p+1)}{\Gamma(\nu+1)\Gamma(k+1)\Gamma(k+2\nu-\lceil\frac\nu2\rceil-\gamma p-\beta p+4)}|z_2|^{k} \\
&\lesssim\sum_{k=0}^{+\infty}\frac{\Gamma(k+\beta p)}{\Gamma(k+1)}|z_2|^{k}\\
&\lesssim(1-|z_2|^2)^{-\beta p}.
\end{align*}
In conclusion, from \eqref{estimate-schur} we get
\begin{equation*}
\int_{\bbH}  |K_\nu(z,w)|f^p(w_1,w_2)\, d\mu_\nu(w)\lesssim (1-|z_1/z_2|^2)^{-\alpha p}(1-|z_2|^2)^{-\beta p }|z_2|^{-\gamma p}
\end{equation*}
if the conditions 
\begin{equation}\label{conditions}
\begin{cases}
\nu-\alpha p>-1\\ \nu-\beta p>-1\\ \nu-\lceil\frac\nu2\rceil-\gamma p+2>-1 \\ 1+\lceil\frac\nu2\rceil-\gamma p<0
\end{cases}
\end{equation}
are satisfied. We obtain the same estimate using the same test function $f$ and $p'$ instead of $p$. Therefore, we can apply Schur's lemma if the conditions \eqref{conditions} are satisfied simultaneously for $p$ and $p'$. This happens if
\begin{align*}
&\alpha\in (0,(\nu+1)/{p})\cap(0,(\nu+1)/{p'}),\qquad \beta\in (0,(\nu+1)/{p})\cap(0,\nu+1/{p'}),\\
&\gamma \in((1+\lceil\nu/2\rceil)/{p},(3+\nu-\lceil\nu/2\rceil)/{p})\cap((1+\lceil\nu/2\rceil)/{p'},(3+\nu-\lceil\nu/2\rceil)/{p'}).
\end{align*}
The parameters $\alpha$ and $\beta$ always vary in a non-trivial range, whereas, in order to have a non-trivial range for $\gamma$, we need that $p\in \Big(2-\frac{\nu-2\lceil\frac\nu2\rceil+2}{\nu-\lceil\frac\nu2\rceil+3}, 2+\frac{\nu-2\lceil\frac\nu2\rceil+2}{1+\lceil\frac\nu2\rceil}\Big)$ and the conclusion follows.

\qed

\medskip
\begin{remark}\emph{
Let $\Phi:\bbD\times\bbD^*\to\bbH$ be the biholomorphism $(w_1,w_2)\mapsto (w_1w_2,w_2)$. 
We observe that the pull-back  of $A^2_\nu(\bbH)$ via $\Phi$, which we denote by $A^2_{\nu}(\bbD\times\bbD^*)$, is the space of holomorphic functions on $\bbD\times\bbD^*$ endowed with the norm
$$
\|f\|^2_{A^2_\nu(\bbD\times\bbD^*)}:= c_\nu\int_{\bbD\times\bbD^*}|f(w_1,w_2)|^2|w_2|^{\nu} (1-|w_1|^2)^\nu(1-|w_2|^2)^{\nu}\, dw<+\infty,
 $$
 where $c_\nu=\frac{(\nu+1)\Gamma(\frac32\nu+3)}{\pi^2\Gamma(\nu+1)\Gamma(\frac\nu2+2)}$. Clearly, the map $f\mapsto \Phi'\cdot f\circ\Phi$, where $\Phi'=\jac_\bbC(\Phi)$, is a surjective isometry from $A^2_\nu(\bbH)$ onto $A^2_\nu(\bbD\times\bbD^*)$. We point out that for $\nu\in(-1,0]$ the functions in $A^2_\nu(\bbD\times\bbD^*)$ are actually holomorphic on $\bbD\times\bbD$.  \newline
\indent  If $p\neq2$ we consider the weighted space $A^p_\nu(\bbD\times\bbD^*,|w_2|^2)$, the space of holomorphic functions on $\bbD\times\bbD^*$ with norm
 \begin{equation*}
 \|f\|_{A^p_\nu(\bbD\times\bbD^*,|w_2|^2)}=c_\nu\int_{\bbD\times\bbD^*}|f(w_1,w_2)|^p|w_2|^{\nu+2} (1-|w_1|^2)^\nu(1-|w_2|^2)^{\nu}\, dw<+\infty.
 \end{equation*}
 Then, $f\mapsto f\circ\Phi$ is a surjective isometry from $A^p_\nu(\bbH)$ onto $A^p_\nu(\bbD\times\bbD^*,|w_2|^2)$.
 } 
\end{remark}
\section{A Hardy space on the Hartogs triangle}\label{hardy-space}
In this section we introduce a candidate Hardy space on the Hartogs triangle. If $\Omega=\{ z:\rho(z)<0\}$ is a
smoothly bounded domain in $\bbC^n$, the Hardy space $H^2(\Omega)$ is defined as 
\begin{equation*}
H^2(\Omega) =\big\{ f\in\Hol(\Omega): \sup_{\varepsilon>0}
\int_{\partial\Omega_\varepsilon} |f|^2 d\sigma_\varepsilon <\infty
\big\} \,,
\end{equation*}
where $\Omega_\varepsilon=\{ z:\rho(z)<-\varepsilon\}$ and $d\sigma_\varepsilon$ is the
induced surface measure on $\partial\Omega_\varepsilon$, the topological boundary of $\Omega_\varepsilon$. Then, $H^2(\Omega)$ can be 
identified with a closed subspace of $L^2(\partial\Omega,d\sigma)$ that we
denote by $H^2(\p\Omega)$.  The
Szeg\H o projection is the orthogonal projection 
\begin{equation*}
S_\Omega : L^2(\p\Omega,d\sigma) \to H^2(\p\Omega) \,;
\end{equation*}
see \cite{Stein-holo}. Such a definition of Hardy space is not suitable for the Hartogs triangle. Having in mind \eqref{hardy-limit}, we introduce in this section a Hardy space $H^2(\bbH)$ as limit of the weighted Bergman spaces $A^2_\nu(\bbH)$ as follows. From \eqref{kernel-general} we get 
\begin{align*}
\begin{split}
K_{-1}((z_1,z_2),(w_1,w_2)):= \lim_{\nu=-1} K_{\nu}((z_1,z_2),(w_1,w_2))=\frac{1}{(z_2\overline w_2-z_1\overline w_1)(1-z_2\overline w_2)}.
 \end{split}
\end{align*}
We want $H^2(\bbH)$ to be the reproducing kernel Hilbert space associated to the kernel $K_{-1}$.

For $(s,t)\in (0,1)\times (0,1)$ we set
\begin{equation*}
\bbH_{st}=\Big\{(z_1,z_2)\in\bbC^2:|z_1|/s<|z_2|<t\Big\}\subsetneq \bbH
\end{equation*}
and
we define the Hardy space $H^2(\bbH)$ as 
\begin{equation*}
H^2(\bbH)=\bigg\{f\in\Hol(\bbH):\sup_{(s,t)\in(0,1)\times(0,1)}\frac{1}{4\pi^2}\int_{d_b(\bbH_{st})}|f|^2\, d\sigma_{st}<+\infty\bigg\},
\end{equation*}
where $d\sigma_{st}$ denotes the induced surface measure on $d_b(\bbH_{st})$, the distinguished boundary of $\bbH_{st}$. We endow $H^2(\bbH)$ with the norm
\begin{equation*}
\|f\|^2_{H^2}:=\!\sup_{(s,t)\in(0,1)\times(0,1)}\frac{1}{4\pi^2}\int_{d_b(\bbH_{st})}\!|f|^2\, d\sigma_{st}=\!\sup_{(s,t)\in(0,1)\times(0,1)} \frac{1}{4\pi^2}\int_0^{2\pi}\int_0^{2\pi}\!|f(st e^{i\theta}, t e^{i\gamma})|^2 st^2 d\theta d\gamma.
\end{equation*}

\begin{remark}\emph{
Unlike the classical setting, we point out that a boundary point, the origin, belongs to all the approximating domains $\bbH_{st}$. Also, our definition of $H^2(\bbH)$ is based on the approximating domain $\mathbb H_{st}$, which depends on two parameters $s,t$, but we could have also used the  approximating domain $\bbH_t$ defined as
\begin{equation*}
\bbH_t=\Big\{(z_1,z_2)\in\bbC^2: |z_1|/t<|z_2|<t\Big\}.
\end{equation*}
The resulting space would be different from the one we considered. Hence, several different approaches are available to define a Hardy space on $\bbH$. A further investigation on this matter would be surely interesting. At the moment our goal is to characterize a reproducing kernel Hilbert space with a prescribed kernel that fits in a one-parameter family of reproducing kernels. As we now see, our space $H^2(\bbH)$ has this property.
}
\end{remark}

\begin{prop} 
The Hardy space $H^2(\bbH)$ is a reproducing kernel Hilbert space with reproducing kernel 
\begin{equation*}
K_{-1}((z_1,z_2),(w_1,w_2))=\frac{1}{(z_2\overline w_2-z_1\overline w_1)(1-z_2\overline w_2)}
\end{equation*}
\end{prop}

\begin{proof}
We first notice that $z_1^jz_2^k\in H^2(\bbH)$ if and only if $j\geq0$ and $j+k+1\geq0$. In fact,
 \begin{align*}
  \|z_1^jz_2^k\|^2_{H^2}&=\sup_{(s,t)\in(0,1)\times(0,1)}\frac{1}{
4\pi^2}\int_0^{2\pi}\int_0^{2\pi}|ste^{i\theta}|^{2j}|te^{i\gamma}|^{2k} st^2\, d\theta d\gamma\\
&=\sup_{(s,t)\in(0,1)\times(0,1)}|s^{2j+1}||t^{2(j+k+1)}|=1.
  \end{align*}
  Therefore, setting $I=\{(j,k): j\geq0 \wedge j+k+1\geq0\}$, $\{ z_1^jz_2^k\}_{(j,k)\in I}$ is an orthonormal basis for $H^2(\bbH)$  . Hence, any function $f\in H^2(\bbH)$ is of the form
  \begin{equation}\label{hardy-power-series}
  f(z_1,z_2)=\sum_{j=0}^{+\infty}\sum_{k=-j-1}^{+\infty}a_{jk}z_1^jz_2^k 
  \end{equation}
and 
\begin{equation*}
\|f\|^2_{H^2(\bbH)}=\sum_{j=0}^{+\infty}\sum_{k=-j-1}^{+\infty}|a_{jk}|^2.
\end{equation*}
For any $z=(z_1,z_2)\in\bbH$ it holds
\begin{align*}
|f(z)|&=\Big|\sum_{j=0}^{+\infty}\sum_{j=-k-1}^{+\infty}\!a_{jk}z_1^jz_2^k\Big|\leq \bigg(\sum_{j=0}^{+\infty}\sum_{j=-k-1}^{+\infty}\!|a_{jk}|^2\bigg)^{\frac12}\bigg(\sum_{j=0}^{+\infty}\sum_{j=-k-1}^{+\infty} |z_1|^{2j}|z_2|^{2k}\bigg)^{\frac12}\leq c_z\|f\|_{H^2},
\end{align*}
where $c_z$ is a positive constant depending on the point $z$. It follows that  point-evaluations are bounded functionals on $H^2(\bbH)$, that is, $H^2(\bbH)$ is a reproducing kernel Hilbert space with kernel given by
\begin{align*}
 \sum_{j=0}^{+\infty}\sum_{k=-j-1}^{+\infty} (z_1\overline w_1)^j(z_2\overline w_2)^k=\frac{1}{(z_2\overline w_2-z_1\overline w_1)(1-z_2\overline w_2)}=K_{-1}((z_1,z_2),(w_1,w_2))
\end{align*}
as we wished to show. 

\end{proof}

\begin{prop}\label{Mf limit}
 Let $f$ be in $H^2(\bbH)$. Then, $f\in A^2_{\nu}$ for all $\nu>-1$ and $\|f\|^2_{H^2} =\lim_{\nu\to-1}\|f\|^2_{A^2_\nu}$.
\end{prop}
\begin{proof}
For any $\nu>-1$ we have
 \begin{align*}
  \|f\|^2_{A^2_\nu}&=C_\nu2^{\frac\nu2}\int_{\bbH}|f(z_1,z_2)|^2|z_2|^\nu(1-|z_1/z_2|^2)^\nu(1-|z_2|^2)^\nu\,dz\\
  &=C_\nu2^{\frac\nu2}\int_{\bbD\times\bbD^*}|f(z_1z_2,z_2)|^2|z_2|^{2+\nu}(1-|z_1|^2)^\nu(1-|z_2|^2)^\nu\, dz\\
  &= C_\nu2^{\frac\nu2}\int_0^1\int_0^1\int_0^{2\pi}\int_{0}^{2\pi}|f(ste^{i\theta},t e^{i\gamma})|^2 (1-s^2)^\nu(1-t^2)^\nu st^{\nu+3}\, d\theta d\gamma ds dt\\
  &\leq \pi^2 2^{\frac\nu2 }C_\nu\|f\|^2_{H^2}  \frac{\Gamma(\frac 12)\Gamma(\nu+1)\Gamma(\frac \nu2+1)\Gamma(\nu+1)}{\Gamma(\nu+\frac 32)\Gamma(\frac 32\nu+2)}.
  \end{align*}
  Taking the limits on both sides and recalling \eqref{Cnu} we obtain that 
  \begin{equation}\label{upper-bound}
  \|f\|^2_{H^2}\geq\lim_{\nu\to-1}\|f\|^2_{A^2_\nu}.
  \end{equation}
Viceversa, given $\nu\in(-1,0)$ and $\varepsilon>0$ there exist $\delta_1,\delta_2>0$ such that
\begin{align*}
 \|f\|^2_{A^2_\nu}&=C_\nu2^{\frac\nu2}\int_{\bbH}|f(z_1,z_2)|^2|z_2|^\nu(1-|z_1/z_2|^2)^\nu(1-|z_2|^2)^\nu\,dz\\
  &=C_\nu2^{\frac\nu2}\int_{\bbD\times\bbD^*}|f(z_1z_2,z_2)|^2|z_2|^{2+\nu}(1-|z_1|^2)^\nu(1-|z_2|^2)^\nu\, dz\\
   &\geq C_\nu2^{\frac\nu2}\int_{\delta_1}^1\int_{\delta_2}^1\int_0^{2\pi}\int_{0}^{2\pi}|f(ste^{i\theta},t e^{i\gamma})|^2 (1-s^2)^\nu(1-t^2)^\nu st^{\nu+3}\, d\theta d\gamma ds dt\\
  &\geq \frac{\pi^2 2^{\frac\nu2} C_\nu}{(\nu+1)^2}(\|f\|^2_{H^2}-\varepsilon)\frac{(1-\delta_1^2)^{\nu+1}(1-\delta_2^2)^{\nu+1}}{(\nu+1)^2}.
\end{align*}
Taking the limit as $\nu\to-1$ we get
\begin{equation*}
\lim_{\nu\to-1}\|f\|^2_{A^2_\nu}\geq  (\|f\|^2_{H^2}-\varepsilon).
\end{equation*}
Since $\varepsilon$ is arbitrary and \eqref{upper-bound} holds, we get that $\|f\|^2_{H^2}\lim_{\nu\to-1}\|f\|^2_{A^2_\nu}$ and the conclusion follows.
\end{proof}

Notice that to any function $f\in H^2(\bbH)$ we can associate a boundary value function $\widetilde f\in L^2(d_b(\bbH))$ defined as 
\begin{equation}\label{boundary-value}
\widetilde f(e^{i\theta}, e^{i\gamma})=\sum_{j=0}^{+\infty}\sum_{k=-j-1}^{+\infty}a_{jk}e^{ij\theta}e^{ik\gamma}.
\end{equation}
The function $\widetilde f$ is a boundary value function for $f$ since, setting
\begin{equation*}
f_{st}(e^{i\theta}, e^{i\gamma}):= f(st e^{i\theta}, t e^{i\gamma}),
\end{equation*}
we have
\begin{align*}
\lim_{(s,t)\to(1,1)}\|\widetilde f- f_{st}\|^2_{L^2(d_b(\bbH))}&=\lim_{(s,t)\to(1,1)}
\sum_{j=0}^{+\infty}\sum_{k=-j-1}^{+\infty}|a_{jk}|^2(1-s^j t^{j+k})^2=0
\end{align*}
by the dominated convergence theorem. Viceversa, any function $g\in L^2(d_b(\bbH)))$ of the form \eqref{boundary-value} automatically extends to a function in $H^2(\bbH)$. Therefore, we can identify the space $H^2(\bbH)$ with the closed space $H^2(d_b(\bbH))\subseteq L^2(d_b(\bbH))$ defined as 
\begin{equation*}
H^2(d_b(\bbH)):=\bigg\{
f(e^{i\theta}, e^{i\gamma})=\sum_{j=0}^{+\infty}\sum_{k=-j-1}^{+\infty}a_{jk}e^{ij\theta}e^{ik\gamma}: \|\{a_{jk}\}\|_{\ell^2}<+\infty\bigg\}.
\end{equation*}
From now on, we call Hardy space both the spaces $H^2(\bbH)$ and $H^2((d_b(\bbH)))$ and we denote by $f$ both the function in $H^2(\bbH)$ and its boundary value in $H^2(d_b(\bbH))$. This should cause no confusion and it will be clear from the context if we are working inside the domain $\bbH$ or on the distinguished boundary $d_b(\bbH)$. Wherever needed we will be more specific about notation and terminology.

We now consider the Szeg\H o projection associated to our Hardy space, that is, the Hilbert space projection operator $S: L^2(d_b(\bbH))\to H^2(d_b(\bbH))$ defined by
\begin{equation*}
f(e^{i\theta}, e^{i\gamma})=\sum_{j,k\in \mathbb Z} a_{jk}e^{ij\theta} e^{ik\gamma}\mapsto Sf(e^{i\theta}, e^{i\gamma}):=\sum_{j=0}^{+\infty}\sum_{k=-j-1}^{+\infty}a_{jk}e^{ij\theta}e^{ik\gamma}.
\end{equation*}

If suitable interpreted, the Szeg\H o projection admits also the integral representation
\begin{align*}
\begin{split}
 f\mapsto Sf(\z_1,\z_2)&=\frac{1}{4\pi^2}\int_{d_b(\bbH)} f(e^{i\theta},e^{i\gamma}) K_{-1}((\z_1,\z_2),(e^{i\theta},e^{i\gamma}))\, d\theta d\gamma
\end{split}
\end{align*}
where $(\z_1,\z_2)$ is any point in $d_b(\bbH)$. However, for our purposes it is enough to consider the Fourier series representation of $Sf$ and we no longer discuss its integral representation.

Since $d_b(\bbH)$ can be identified with the $2$-dimensional torus, from the classical theory of Fourier series on the torus the boundedness of $S$ will follow from the boundedness of the Fourier multiplier operator associated to the multiplier
\begin{equation}\label{szego-proj-2}
m(j,k)= \frac{1+\sgn(j)}{2}\cdot\frac{1+\sgn(j+k+1)}{2}.
\end{equation}
We recall that a Fourier multiplier operator on the $2$-dimensional torus is an operator of the form 
\begin{equation*}
f\mapsto \sum_{(j,k)\in\bbZ^2} m(j,k)\widehat f(j,k) e^{ij\theta}e^{ik\theta},
\end{equation*}
where $\widehat f(j,k)$ denotes the $(j,k)$-Fourier coefficient of the function $f$ and $\{m(j,k)\}_{(j,k)\in\bbZ^2}$ is a bounded sequence.
We now prove Theorem \ref{main-2}.
\proof[Proof of Theorem \ref{main-2}]
The proof follows from classical results in harmonic analysis. Indeed, let us consider on $L^2(\bbR^2)$ the Fourier multiplier operator associated to the multiplier
 \begin{equation*}
 \widetilde m(\xi,\eta)=\frac{1+\sgn(\xi)}{2}\cdot\frac{1+\sgn(\xi+\eta+1)}{2},
 \end{equation*}
 that is, the operator $Tf:= \mathcal F^{-1}(\widetilde m\mathcal F f)$,  where $\mathcal F$ and $\mathcal F^{-1}$ denote the Fourier transform on $\mathbb R^2$ and its inverse respectively.
This operator is well-defined on the class of smooth compactly supported functions and  extends to a bounded operator $T:L^p(\bbR^2)\to L^p(\bbR^2)$ for any $p\in(1,+\infty)$ by standard results on the Hilbert transform. By transference, see \cite{Grafakos-cl}, we obtain that the Fourier multiplier operator associated to $\widetilde m|_{\bbZ^2}$, that is, the Fourier multiplier operators associated to the multiplier defined in \eqref{szego-proj-2}, extends to a bounded operator $L^p(\partial\bbD\times\partial \bbD)\to L^p(\partial \bbD\times\partial \bbD)$ for any $p\in(1,+\infty)$. Therefore, the Szeg\H o projection  extends to a bounded operator $S: L^p(d_b(\bbH))\to L^p(d_b(\bbH))$ for any $p\in(1,+\infty)$ as we wished to show.
\qed

\medskip

We conclude the section observing that there exists a surjective isometry from $H^2(\bbH)$ and $H^2(\bbD\times\bbD)$, the Hardy space on the bidisc. This latter space can be described as the space of functions $f\in\Hol(\bbD\times\bbD)$, $f(z_1,z_2)=\sum_{j,k\geq0} a_{jk}z_1^jz_2^k$ such that 
\begin{equation*}
\|f\|^2_{H^2(\bbD\times\bbD)}:=\sum_{j=0}^{+\infty}\sum_{k=0}^{+\infty}|a_{jk}|^2<+\infty.
\end{equation*}

\begin{prop}\label{hardy-iso}
 Let $\Phi:\bbD\times\bbD^*\to\bbH$ be the biholomorphic map $(w_1,w_2)\mapsto(w_1w_2,w_2)$ and set $\Phi'=\jac_{\bbC}(\Phi)$. Then, the map
 \begin{equation*}
 f\mapsto \Phi'\cdot f\circ\Phi
 \end{equation*}
 is a surjective isometry from  $H^2(\bbH)$  onto $H^2(\bbD\times\bbD)$.
\end{prop}

\begin{proof}
Given $f\in H^2(\bbH)$, from \eqref{hardy-power-series} we get 
\begin{equation*}
\Phi'\cdot f\circ \Phi(w_1,w_2)=\sum_{j=0}^{+\infty}\sum_{k=-j-1}^{+\infty}a_{jk}w_1^jw_2^{k+j+1}= \sum_{j=0}^{+\infty}\sum_{k=0}^{+\infty}\widetilde a_{jk}w_1^jw_2^{k},
\end{equation*}
where $\widetilde a_{jk}= a_{j(k-j-1)}$. Hence, $\Phi'\cdot f\circ \Phi\in\Hol(\bbD\times\bbD)$ and 
\begin{equation*}
\|\Phi'\cdot f\circ \Phi\|^2_{H^2(\bbD\times\bbD)}=\sum_{j=0}^{+\infty}\sum_{k=0}^{+\infty}|\widetilde a_{jk}|^2= \sum_{j=0}^{+\infty}\sum_{k=-j-1}^{+\infty}|a_{jk}|^2=\|f\|^2_{H^2(\bbH)}.
\end{equation*}
Therefore, the map $
 f\mapsto \Phi'\cdot f\circ\Phi$ is an isometry.  The map is clearly surjective and the inverse is given by $(\Phi^{-1})'\cdot f\circ \Phi^{-1}$.
\end{proof}

\begin{remark}\emph{
We point out that the surjective isometry $f\mapsto \Phi'\cdot f\circ\Phi$ does not coincide with the expected result. In fact, let $\Omega_1,\Omega_2\subseteq\bbC$ be bounded domains with $\mathcal C^{\infty}$ boundaries  and let $\varphi:\Omega_1\to\Omega_2$ be a biholomorphic mappings. Then, then map $f\mapsto\sqrt{\varphi'}\cdot f\circ\varphi$ is an isometric isomorphism between the Hardy spaces $H^2(\Omega_1)$ and $H^2(\Omega_2)$ (\cite[Chapter 12]{Bell-book}). In order to prove this result one has to prove that $\sqrt{\varphi'}$ is a well-defined holomorphic function on $\Omega_1$. Moreover, since $\Omega_1$ and $\Omega_2$ are smooth bounded domains, there is no ambiguity in defining the Hardy spaces on these domains. In our setting we would expect the isomorphism to be of the form $f\mapsto \sqrt{\Phi'}\cdot f\circ\Phi$ where $\Phi,\Phi'$ are defined as in Proposition \ref{hardy-iso}. However, this is not the case. The  culprit may be the fact that there is no standard definition for the Hardy space on $\bbH$, whereas there is for the Hardy space on $\bbD\times\bbD$. With our definition of $H^2(\bbH)$ the expected isomorphism fails. Moreover, notice that $(\Phi^{-1})'(z_1,z_2)=z_2^{-1}$, thus the factor $\sqrt{(\Phi^{-1})'}$ in the expected isomorphism between $H^2(\bbH)$ and $H^2(\bbD\times\bbD)$ is not even well-defined. 
}
\end{remark}

\section{Weighted Dirichlet and Dirichlet Spaces}\label{dirichlet-space}
In this section we enlarge our family of function spaces defining some weighted Dirichlet spaces $\mathcal D_\nu$, for $\nu\in(-2,-1)$ and a candidate Dirichlet space $\mathcal D$ for $\nu=-2$.

Given $f(z_1,z_2)=\sum_{j=0}^{+\infty}\sum_{k=-\infty}^{\infty} a_{jk}z_1^jz_2^k\in \Hol(\bbH)$, we define the functions
\begin{align*}
 &f_1(z_1,z_2):=\sum_{j=1}^{+\infty}\sum_{k\neq-j}j(j+k)a_{jk}z_1^jz_2^{k-1};\\
 &f_2(z_1,z_2):=\sum_{k\in\mathbb Z\backslash\{0\}}k a_{0k}z_2^{k-1};\\ 
&f_3(z_1,z_2):=\sum_{j=1}^{+\infty}j a_{j(-j)}z_1^j z_2^{-j-1}
 \end{align*}
and define the operator
\begin{align}\label{operators-T}
Tf(z_1,z_2):=|z_2|(1-|z_1/z_2|^2)(1-|z_2|^2)f(z_1,z_2).
\end{align}

Then, we prove the following result. An analogous result in the setting of the polydisc was proved in \cite[Theorem B]{zhu-poly}.

\begin{thm}\label{bergman-derivative}
Let $\nu>-1$. A function $f$ belongs to $A^2_\nu(\bbH)$ if and only if $Tf_j, j=1,2,3$, belong to $L^2_\nu(\bbH)$.
\end{thm}
\begin{proof}
We first recall \eqref{expansion-f} and \eqref{expansion-norm}, that is, a function $f\in A^2_\nu(\bbH)$ if and only if 
\begin{equation*}
  f(z_1,z_2)=\sum_{j=0}^{+\infty}\sum_{k>-j-\frac\nu2-2}a_{jk}z_1^jz_2^k
 \end{equation*}
and
\begin{equation*}
\|f\|^2_{A^2_\nu}=\frac{(\nu+1)\Gamma(\nu+1)\Gamma(\frac32\nu+3)}{\Gamma(\frac\nu2+2)}\sum_{j=0}^{+\infty}\sum_{k>-j-\frac\nu2-2}\frac{\Gamma(j+1)\Gamma(j+k+\frac\nu2+2)}{\Gamma(j+\nu+2)\Gamma(j+k+\frac32\nu+3)}|a_{jk}|^2
\end{equation*}

Then, 
\begin{align*}
 &\|Tf_1\|^2_{L^2_\nu}= C_\nu 2^{\frac\nu2}\int_{\bbH}\Big|\sum_{j=1}^{+\infty}\sum_{k\neq-j}j(j+k)a_{jk}z_1^jz_2^{k-1}\Big|^2 |z_2|^{2+\nu}(1-|z_1/z_2|^2)^{2+\nu}(1-|z_2|^2)^{2+\nu}\, dz\\
 &\,\,\,=4\pi^{2}2^{\frac\nu2} C_\nu\sum_{j=1}^{+\infty}\sum_{k\neq-j}j^2(j+k)^2|a_{jk}|^2\int_0^1\int_0^{\rho}r^{2j+1}\rho^{2k+\nu+1}(1-(r/\rho)^2)^{2+\nu}(1-\rho^2)^{2+\nu}\, drd\rho\\
 &\,\,\,=\pi^{2}2^{\frac\nu2}C_\nu\sum_{j=1}^{+\infty}\sum_{k\neq-j}j^2(j+k)^2|a_{jk}|^2\int_0^1r^j(1-r)^{2+\nu}\,dr\int_0^1 \rho^{j+k+\frac\nu2+1}(1-\rho)^{2+\nu}\, d\rho\\
 &<+\infty
 \end{align*}
 if and only if $\nu>-3$ and $j+k+\frac\nu2+1>-1$. In particular,
 
 \begin{equation}\label{norm-T1}
 \|Tf_1\|^2_{L^2_\nu}\approx\sum_{j=1}^{+\infty}\sum_{\substack{k\neq-j\\k>-j-\frac\nu2-2}}j^2(j+k)^2 \frac{\Gamma(j+1)\Gamma(j+k+\frac\nu2+2)}{\Gamma(j+\nu+4)\Gamma(j+k+\frac32\nu+5)}|a_{jk}|^2.
\end{equation}

Similarly, for $Tf_2$ we get
\begin{align*}
 \|Tf_2\|^2_{L^2_\nu}&=4\pi^22^{\frac\nu2}C_\nu\sum_{k\in\bbZ\backslash\{0\}}k^2|a_{0k}|^2\int_0^1\int_0^{\rho}r\rho^{2k+\nu+1}(1-(r/\rho)^2)^{2+\nu}(1-\rho^2)^{2+\nu}\, drd\rho\\
 &=\pi^{2}C_\nu\sum_{k\in\bbZ\backslash\{0\}}k^2|a_{0k}|^2\int_0^1(1-r)^{2+\nu}\,dr\int_0^1\rho^{k+\frac\nu2+1}(1-\rho)^{2+\nu}\, d\rho\\
 &<+\infty
 \end{align*}
 if and only if $\nu>-3$ and $k+\frac\nu2+1>-1$.
 In particular,
 \begin{equation}\label{norm-T2}
 \|Tf_2\|^2_{L^2_\nu}\approx\sum_{\substack{k\in\bbZ\backslash\{0\}\\k>-\frac\nu2-2}}k^2\frac{\Gamma(k+\frac\nu2+2)}{\Gamma(k+\frac32\nu+5)}|a_{0k}|^2.
\end{equation}
At last,
\begin{align*}
 \|Tf_3\|^2_{L^2_\nu}&=4\pi^{2}2^{\frac\nu2}C_\nu\sum_{j=1}^{+\infty}j^2|a_{j(-j)}|^2\int_0^1\int_0^\rho r^{2j+1}\rho^{-2j+\nu+1}(1-(r/\rho)^2)^{2+\nu}(1-\rho^2)^{2+\nu}\, drd\rho\\
 &=\pi^{2}2^{\frac\nu2}C_\nu\sum_{j=1}^{+\infty}j^2|a_{j(-j)}|^2\int_0^1r^j(1-r)^{2+\nu}\,dr\int_0^1\rho^{\frac\nu2+1}(1-\rho)^{2+\nu}\, d\rho\\
 &<+\infty
 \end{align*}
 if and only if $\nu>-3$ and $\frac\nu2+2>-1$. Hence,
\begin{equation}\label{norm-T3}
 \|Tf_3\|^2_{L^2_\nu}\approx\sum_{j=1}^{+\infty}j^2\frac{\Gamma(j+1)}{\Gamma(j+\nu+4)}|a_{j(-j)}|^2.
\end{equation}
The conclusion follows comparing \eqref{expansion-f} and \eqref{expansion-norm} with \eqref{norm-T1},\eqref{norm-T2} and \eqref{norm-T3}.
\end{proof}
From the previous proposition we deduce that we can we can endow $A^2_\nu$ with the equivalent norm
\begin{equation*}
\|f\|_{A^2_\nu}=|a_{00}|+\|Tf_1\|_{L^2_\nu}+\|Tf_2\|_{L^2_\nu}+\|Tf_3\|_{L^2_\nu}.
\end{equation*}
Moreover, we can use this norm to define some new spaces for $-2<\nu<-1$, namely, we define the weighted Dirichlet space $\mathcal D_\nu$ as 
\begin{equation*}
\mathcal D_\nu =\big\{f\in\Hol(\bbH): \|f\|_{*}<+\infty\big\}, 
\end{equation*}
where
\begin{equation}\label{norm-Dnu}
\|f\|_{*}:=|a_{00}|+\|Tf_1\|_{L^2_\nu}+\|Tf_2\|_{L^2_\nu}+\|Tf_3\|_{L^2_\nu}.
\end{equation}
We have the following result. 
\begin{thm}
 There exists an inner product $\left<\cdot,\cdot\right>_{\mathcal D_\nu}$ on $\mathcal D_\nu$ such that $(\mathcal D_\nu,\|\cdot\|_{\mathcal D_\nu})$ is a reproducing kernel Hilbert space with kernel 
\begin{equation*}
K_{\nu}(z,w)=c_\nu\frac{(z_2\overline w_2)^{-1}}{(1-\frac{z_1\overline w_1}{z_2\overline w_2})^{\nu+2}}F\Big(\frac32\nu+2,1;\frac\nu2+1,z_2\overline w_2\Big).
\end{equation*}
where $c_\nu= \frac{(\frac\nu2+1)}{(\frac32\nu+2)}$.
\end{thm}
\begin{proof}
From \eqref{operators-T} and \eqref{norm-Dnu} we get that any function $f\in \mathcal D_\nu$ is of the form
\begin{equation*}
f(z_1,z_2)=\sum_{j=0}^{+\infty}\sum_{k>-j-\frac\nu2-2}a_{jk}z_1^jz_2^k.
\end{equation*}
Moreover, notice that the formula for the $A^2_\nu(\bbH)$ norm in terms of the Laurent series coefficients, that is, formula \eqref{expansion-norm}, is meaningful for $-2<\nu<-1$ as well. Therefore, we endow $\mathcal D_\nu$ with the inner product
\begin{equation*}
\left<f,g\right>_{\mathcal D_\nu}:=\frac{(\nu+1)\Gamma(\nu+1)\Gamma(\frac32\nu+3)}{\Gamma(\frac\nu2+2)}\sum_{j=0}^{+\infty}\sum_{k>-j-\frac\nu2-2}\frac{\Gamma(j+1)\Gamma(j+k+\frac\nu2+2)}{\Gamma(j+\nu+2)\Gamma(j+k+\frac32\nu+3)}a_{jk}\overline {b_{jk}}
\end{equation*}
where 
\begin{equation*}
g(z_1,z_2):=\sum_{j=0}^{+\infty}\sum_{k>-j-\frac\nu2-2}b_{jk}z_1^jz_2^k.
\end{equation*}
From the properties of the Gamma function we deduce that the norm induced by this inner product is equivalent to the norm \eqref{norm-Dnu}. By means of the Cauchy-Schwartz inequality, similarly to the proof of Proposition \ref{Bergman-kernel}, we obtain that $\mathcal D_\nu$ is a reproducing kernel Hilbert space and the reproducing kernel $K_{\nu}$ is given by
\begin{align*}
 K_{\nu}((z_1,z_2), (w_1,w_2)) &= \frac{(\frac\nu2+1)}{(\frac32\nu+2)}\frac{(z_2\overline w_2)^{-1}}{(1-\frac{z_1\overline w_1}{z_2\overline w_2})^{\nu+2}}F\Big(\frac32\nu+2,1;\frac\nu2+1,z_2\overline w_2\Big).
\end{align*}
as we wished to show.
\end{proof}

\medskip

Finally we define our candidate Dirichlet space $\mathcal D$, which corresponds to the case $\nu=-2$, as  
\begin{equation*}
 \mathcal D=\big\{f\in\Hol(\bbH): \|f\|_{\sharp}<+\infty\big\}
\end{equation*}
where
\begin{equation}\label{norm-Dirichlet}
\|f\|_{\sharp}:= |a_{00}|+\|Tf_1\|_{L^2_{-2}}+\|Tf_2\|_{L^2_{-2}}+\|Tf_3\|_{L^2_{-2}}.
\end{equation}
We recall that the space $L^2_{-2}(\bbH)$ is endowed with the norm \eqref{Lp-norm}, that is,
\begin{equation*}
\|f\|^2_{L^2_{-2}}\approx\int_{\bbH}|f(z_1,z_2)|^2|z_2|^{-2}(1-|z_1/z_2|^2)^{-2}(1-|z_2|^2)^{-2}\, dz,
\end{equation*}
and we observe that the measure 
\begin{equation*}
(1-|z_1/z_2|^2)^{-2}(1-|z_2|^2)^{-2}|z_2|^{-2}\,dz= K(z,z)\, dz
\end{equation*}
is invariant for $\Aut(\bbH)$. The automorphisms of the Hartogs triangle are completely characterized and  are described in the following proposition. We refer the reader to \cite{L,CX,K,K1}. 

\begin{prop}
 A function $\Psi\in \Aut(\bbH)$ if and only if $\Psi(z_1,z_2)=(\tilde z_1,\tilde z_2)$ where
\begin{align*}
 &\tilde z_1=z_2\varphi\Big(\frac{z_1}{z_2} \Big),\qquad\quad\!\!\! \varphi\in\Aut(\bbD)\\
 &\tilde z_2= c z_2,\qquad\quad\quad\quad  c\in\bbC,|c|=1.
\end{align*}
\end{prop}
The following proposition holds.
\begin{prop}\label{tau}
Let $d\tau$ be the measure 
\begin{equation*}
d\tau:=K(z,z)\,dz=|z_2|^{-2}(1-|z_1/z_2|^2)^{-2}(1-|z_2|^2)^{-2}\,dz.
\end{equation*}
Then, 
 \begin{equation*}
 \int_{\bbH}f\, d\tau=\int_{\bbH} (f\circ\Psi)\, d\tau
 \end{equation*}
 for all $\Psi\in\Aut(\bbH)$.
\end{prop}

\begin{proof}[Proof of Proposition \ref{tau}]
It suffices to show that 
\begin{equation*}
|K(\Psi(z),\Psi(z))||\Psi'(z)|^2\,dz=K(z,z)\, dz
\end{equation*}
for all $\Psi\in\Aut(\bbH)$. This equality holds if and only if
\begin{equation*}
\Big(1-\big|\varphi\big(z_1/z_2\big)\big|\Big)^{-2}\Big|\varphi'\big(z_1/z_2\big)\Big|^{2}\, dz=\Big(1-\big|z_1/z_2\big|\Big)^{-2}\, dz 
\end{equation*}
for all $\varphi\in\Aut(\bbD)$. The  conclusion now follows from the invariance of the measure $(1-|\eta|^2)^{-2}\, d\eta$ with respect to $\varphi$ on the unit disc.
\end{proof}
From \eqref{operators-T} and \eqref{norm-Dirichlet} we get that any function $f\in\mathcal D$ is of the form
\begin{equation*}
f(z_1,z_2)=\sum_{j=0}^{+\infty}\sum_{k=-j}a_{jk}z_1^jz_2^k.
\end{equation*}
We have the following result.
\begin{thm}\label{inner product-Dirichlet}
There exists an inner product $\left<\cdot,\cdot,\right>_{\mathcal D}$ on $\mathcal D$ such that $(\mathcal D,\|\cdot\|_{\mathcal D})$ is a a reproducing kernel Hilbert space with kernel
\begin{equation*}
 K_{\mathcal D}((z_1,z_2),(w_1,w_2))=\frac{1}{z_1\overline w_1}\log\Big(\frac{1}{1-(z_1\overline w_1)/(z_2\overline w_2)}\Big)\log\Big(\frac{1}{1-z_2\overline w_2}\Big)
\end{equation*}
\end{thm}
\begin{proof}
We endow $\mathcal D$ with the inner product
\begin{equation*}
\left<f,g\right>_{\mathcal D}:=\sum_{j=0}^{+\infty}\sum_{k=-j}^{+\infty}(j+1)(j+k+1)a_{jk}\overline{b_{jk}}.
\end{equation*}
From the property of the Gamma function we get that the norm induced by this inner product is equivalent to the norm $\|\cdot\|_{\sharp}$. Similarly to the proof of Proposition \ref{Bergman-kernel}, we prove that $\mathcal D$ is a reproducing kernel Hilbert space by means of the Cauchy--Schwartz inequality and the reproducing kernel of $\mathcal D$ with respect to $\left<\cdot,\cdot\right>_{\mathcal D}$ is given by
\begin{align*}
 K_{\mathcal D}((z_1,z_2),(w_1,w_2))&=\sum_{j=0}^{+\infty}\sum_{k=-j}^{+\infty}\frac{(z_1\overline w_1)^j(z_2\overline w_2)^k}{(j+1)(j+k+1)}\\
&=\frac{1}{z_1\overline w_1}\log\Big(\frac{1}{1-(z_1\overline w_1)/(z_2\overline w_2)}\Big)\log\Big(\frac{1}{1-z_2\overline w_2}\Big)
\end{align*}
as we wished to show.
\end{proof}

Finally, we prove that there exists a surjective isometry from $\mathcal D$ onto $\mathcal D(\bbD^2)$, the Dirichlet space on the bidisc. Following \cite{AMPS2}, we recall that a holomorphic function $f(w_1,w_2)=\sum_{j=0}^{+\infty}\sum_{k=0}^{+\infty}a_{jk}w_1^jw_2^k$ is in $\mathcal D(\bbD^2)$ if and only if 
\begin{equation*}
\|f\|^2_{\mathcal D(\bbD^2)}=\sum_{j=0}^{+\infty}\sum_{k=0}^{+\infty}|a_{jk}|^2(j+1)(k+1)<+\infty.
\end{equation*}
\proof[Proof of Theorem \ref{main-3}]
Let $f(z_1,z_2)=\sum_{j=0}^{+\infty}\sum_{k=-j}^{+\infty}a_{jk}z_1^jz_2^k\in \mathcal D$ and consider the map $f\mapsto\widetilde f:= f\circ\Phi$  where $\Phi(w_1,w_2)=(w_1w_2,w_2)$ is the biholomorphic map $\Phi:\mathbb D\times\bbD^*\to\bbH$, $(w_1,w_2)\mapsto(w_1w_2,w_2)$. Then,
\begin{align*}
\widetilde f(w_1,w_2)&=\sum_{j=0}^{+\infty}\sum_{k=-j}^{+\infty} a_{jk}(w_1w_2)^jw_2^k=\sum_{j=0}^{+\infty}\sum_{k=0}^{+\infty}a_{j(k-j)}w_1^j w_2^{k}=\sum_{j=0}^{+\infty}\sum_{k=0}^{+\infty}\widetilde a_{jk}w_1^j w_2^{k},
\end{align*}
where we set $\widetilde a_{jk}=a_{j(k-j-1)}$. Therefore, $\widetilde f\in \Hol(\bbD^2)$ and
\begin{align*}
\|\widetilde f\|^2_{\mathcal D(\bbD^2)}&=\sum_{j=0}^{+\infty}\sum_{k=0}^{+\infty}|\widetilde a_{jk}|^2(j+1)(k+1)\\
&=\sum_{j=0}^{+\infty}\sum_{k=0}^{+\infty}|a_{j(k-j)}|^2(j+1)(k+1)\\
&=\sum_{j=0}^{+\infty}\sum_{k=-j}^{+\infty}|a_{jk}|^2(j+1)(j+k+1)\\
&=\|f\|^2_{\mathcal D}
\end{align*}
as we wished to show. The map $f\mapsto\widetilde f$ is clearly surjective and the inverse map is given by $g\mapsto  g\circ\Phi^{-1}$ where $g\in \mathcal D(\bbD^2)$. 
\qed

\section{Concluding Remarks}\label{final-remarks}

We think that the family of holomorphic functions spaces we introduced is interesting and worth investigating. The results we prove in this paper are just a first step for this investigation and the techniques we use are classical and not excessively complicated. The major contribution of the paper is the very definition of the family of spaces itself. A few comments are needed.

Our proofs heavily rely on the Hilbert space setting and on the Laurent expansion of the functions involved. If we want to develop our spaces in a $L^p$ setting, $p\neq2$, we need a different approach, especially for the definition of the weighted Dirichlet and Dirichlet spaces. Let $\nu\in(-2,-1)$ and consider the weighted Dirichlet space on the disc $\mathcal D_\nu(\bbD)$, that is, the space of holomorphic functions on $\mathbb D$ endowed with the norm
\begin{equation*}
\|f\|^2_{\nu}=\Big\|\sum_{j\geq0} a_jz^j\Big\|^2_{\nu}:=\sum_{j\geq 0} \frac{\Gamma(j+1)\Gamma(\nu+2)}{\Gamma(j+\nu+2)}|a_j|^2.
\end{equation*}
Then, $\mathcal D_\nu(\bbD)$ is a reproducing kernel Hilbert space with kernel $K_\nu(z,w)=(1-z\overline w)^{-\nu-2}$. It is easily seen that $\|f\|_\nu<+\infty$ if and only if 
\begin{equation*}
\int_\bbD |(1-|z|^2)\mathcal Rf(z)|^2(1-|z|^2)^\nu\, dz<+\infty,
\end{equation*}
where $\mathcal R$ is the radial derivative, that is, $\mathcal Rf(z)=\sum_{j=1}^{+\infty}ja_jz^j$. Therefore, we would like to  define the spaces $\mathcal D_\nu(\bbH)$ in terms of a natural differential operator on $\bbH$ which plays the role of the radial derivative in the unit disc. This approach would be better suited for an investigation in a non-Hilbert setting.

Another point worth of further investigation is the definition of the Hardy space $H^2(\bbH)$. As we have seen, our Hardy space arise in a very natural way, but other approaches are available. Moreover, our space $H^2(\bbH)$ does not see the pathological geometry of the Hartogs triangle. It would be interesting to define a different, but still natural, Hardy space which keeps track of the origin and compare it with our space.

Finally, it would be interesting to use our approach to define some new function spaces on generalizations of the Hartogs triangle, such as, for instance, the fat or thin Hartogs triangles.

\bigskip

{\em Acknowledgments.}  We warmly thank the anonymous referees for useful comments and remarks 
  that improved the presentation of this paper. 
 \bibliography{Hartogsbib}
 \bibliographystyle{amsalpha}

\end{document}